\documentclass{article}
\usepackage{graphicx} 
\usepackage[utf8]{inputenc}
\usepackage{amssymb,amsthm,amsmath,amstext,amscd,bbm}
\usepackage{graphicx} 
\usepackage[utf8]{inputenc}
\usepackage{amssymb,amsthm,amsmath,amstext,amscd}
\usepackage{xcolor}
\usepackage[normalem]{ulem}
\usepackage{tikz-cd}
\usepackage{indentfirst}
\usepackage{comment}
\usepackage{hyperref}
\allowdisplaybreaks
\numberwithin{equation}{section}\theoremstyle{plain}

\newtheorem{theorem}{Theorem}

\newtheorem{corollary}[theorem]{Corollary}
\newtheorem{proposition}[theorem]{Proposition}
\newtheorem{lemma}[theorem]{Lemma}

\newtheorem{remark}[theorem]{Remark}

\newcommand{\der}{{\rm Der}}

\newcommand{\aut}{{\rm Aut}}

\DeclareMathOperator{\centro}{\mathcal{Z}}
\DeclareMathOperator{\centroT}{\mathcal{Z}_{\sigma}}

\newcommand{\dk}{\mathcal{D}_k}
\newcommand{\dfk}{\mathfrak{d}^{(k)}}

\newcommand{\QPq}{\mathcal{A}_q}
\newcommand{\QPs}{\mathcal{A}_{-1}}

\newcommand{\Ct}{\mathcal{C}}
\newcommand{\Cts}{\mathcal{C}^{\sigma}}

\title{Isotropy Groups of $\sigma$-Derivations on the Quantum Plane}
\author{R. Cavalheiro and R. Baltazar }

\date{}

\begin{document}

\maketitle

\begin{abstract}
    
Let $\Bbbk$ be an algebraically closed field of characteristic
zero, and let $\mathcal A_q=\Bbbk_q[x,y]$ be the quantum plane,
with $q\in\Bbbk^\ast\setminus\{1\}$. For a fixed automorphism $\sigma$ of $\mathcal A_q$, we study $\sigma$-derivations and their isotropy groups under the conjugation action of the centralizer of $\sigma$ in $\operatorname{Aut}(\mathcal A_q)$. For $q\neq\pm1$, we use Jordan's classification of skew derivations for toric automorphisms, which generalizes the description of Almulhem and Brzeziński for the quantum plane, to describe these groups by character equations on the torus
$(\Bbbk^\ast)^2$, obtaining an arithmetic criterion for
finiteness. In this case, we also prove that, for every fixed
$\sigma$, every finite subgroup of the torus occurs, up to
isomorphism, as the isotropy group of a $\sigma$-derivation,
including when $q$ is a root of unity. For $q=-1$, we classify the $\sigma$-derivations for both toric and flip automorphisms $\sigma$ and describe their isotropy groups. In particular, we explicitly determine the isotropy groups of ordinary derivations of $\mathcal A_{-1}$.

\end{abstract}

\section{Introduction}
Let $\Bbbk$ be an algebraically closed field of characteristic zero, let $A$ be a $\Bbbk$-algebra and let $\sigma\in\aut(A)$ be a fixed automorphism. A $\Bbbk$-linear map $\delta:A\to A$ is called a \emph{$\sigma$-derivation} if it satisfies the twisted Leibniz rule
\[
\delta(ab)=\delta(a)\sigma(b)+a\delta(b),
\quad a,b\in A.
\]

We denote by $\der_\sigma(A)$ the set of all $\sigma$-derivations of $A$. As in the ordinary case, there is a natural class of inner examples: for each $w\in A$, the map
\[
[w,\_]_\sigma:A\to A,
\quad
[w,b]_\sigma=w\sigma(b)-bw,
\]
is a $\sigma$-derivation, called the \emph{inner $\sigma$-derivation} induced by $w$. Throughout the paper, we also denote this inner $\sigma$-derivation by $I_w:=[w,\_]_\sigma$.

The isotropy of derivations, and more generally of $\sigma$-derivations, offers a natural perspective on the symmetries of algebraic operators. Although isotropy subgroups arise in several contexts, their structure is still far from being completely understood. They appear, for instance, in the study of symmetries of foliations, in questions related to Hochschild cohomology, and in
the investigation of invariants under group actions in Algebra and Geometry \cite{PanBalta, PanMendes, YanShamsuddin}.

In this paper, for a fixed automorphism $\sigma\in\aut(A)$ and a
$\sigma$-derivation $\delta\in\der_\sigma(A)$, we consider the isotropy group
of $\delta$ in the centralizer of $\sigma$:
\[
\aut_\delta(A)
=
\{\rho\in\aut(A)\mid \rho\sigma=\sigma\rho
\text{ and } \rho\delta=\delta\rho\}.
\]

When $\sigma=\operatorname{id}$, this recovers the usual isotropy group of an ordinary derivation. In general, however, the conditions require compatibility of $\rho$ both with the $\sigma$-derivation structure and with the twisting automorphism $\sigma$. Indeed, it is straightforward to check that $\rho\,\delta\,\rho^{-1}\in \der_{\,\rho\sigma\rho^{-1}}(A)$, for every $\rho\in \aut(A)$. That is, requiring that $\rho$ commutes with $\sigma$ guarantees that the action of $\rho$ preserves the structure of the $\sigma$-derivation.

J. Alev and M. Chamarie (\cite[Proposition 1.4.4]{AlevChamarie}), in the particular case of the quantum plane $\QPq = \Bbbk_{q}[x,y]$, proved that, for $q\neq\pm1$, the automorphism group is purely toric: $\rho(x)=\mu_1x$, $\rho(y)=\mu_2y$, with $\mu_1, \mu_2 \in \Bbbk^\ast$. Consequently, when studying isotropy groups of $\sigma$-derivations on the quantum plane with $q^2 \neq 1$, it is enough to consider toric automorphisms $\rho$ satisfying $\rho \delta = \delta \rho$. 

The case $q=-1$ exhibits a different behavior. In this case, by (\cite[Proposition 1.4.4]{AlevChamarie}), we have: 
\[
\aut(\QPs)
\cong
(\Bbbk^\ast\times\Bbbk^\ast)\rtimes\langle\tau\rangle,
\]
where $\tau(x) = y$, $\tau(y) = x$. Thus, in the case $q=-1$, the isotropy problem must be considered in the
centralizer:
\[
\Cts = C_{\aut(\QPs)}(\sigma) = \{ \rho \in \aut(\QPs) \mid \rho \sigma = \sigma \rho \}.
\]

In \cite{SBVA}, A. Santana, R. Baltazar, R. Vinciguerra and W. Araujo analyzed the isotropy of ordinary derivations on the quantum plane under the assumption $q^2\neq1$, a setting that includes the case where $q$ is a root of unity. They showed that the isotropy group associated with a derivation can be trivial, cyclic, finite of prescribed cardinality, or infinite. Their approach relates isotropy groups to arithmetic
conditions and algebraic curves defined by character
equations.

M. Almulhem and T. Brzeziński developed in \cite{Brz2018} a general theory of $\sigma$-derivations on generalized Weyl algebras. For our purposes, one of the most relevant consequences is their Theorem 6.2, where they specialize this theory to the quantum disc and to the quantum plane. In particular, viewing the
quantum plane as the generalized Weyl algebra
\[
\mathcal A_q\simeq \Bbbk[h](h,\varphi),
\quad \varphi(h)=qh,
\]
they obtain an explicit classification of all $\sigma_\mu$-derivations when $q \in  \Bbbk^\ast$ is not a root of unity, where $\sigma_\mu$ is the automorphism given by 
\[
\sigma_\mu(x)=\mu^{-1}x \quad \text{and} \quad \sigma_\mu(y)=\mu y.
\]

More recently, D. A. Jordan \cite{Jordan2025} generalized
this classification for the quantum plane to arbitrary
toric automorphisms, including the root-of-unity case,
as part of his study of skew derivations of quantum
tori and quantum affine spaces.

In Section \ref{scharacterization} we study the case $q^2\neq1$. In this setting, every automorphism of the quantum plane $\mathcal A_q$ is toric. We use Jordan's recent classification of skew derivations for toric
automorphisms \cite{Jordan2025}. More precisely, we translate
Jordan's result into our convention and organize the resulting decomposition in a form adapted to the isotropy problem, separating each $\sigma$-derivation into an inner component and explicit non-inner families (see Theorem \ref{classification}). This extends, in the case of the quantum plane, the classification obtained by Almulhem and Brzeziński in \cite[Theorem 6.2]{Brz2018} for the automorphisms $\sigma_\mu$, including the case where $q\neq\pm1$ is a root of unity.

In the same Section \ref{scharacterization} we study the isotropy groups of the $\sigma$-derivations when $q \neq \pm 1$. We explicitly describe the isotropy groups of each family appearing in Theorem \ref{classification}. Since these families form independent $\aut(\mathcal A_q)$-submodules, the isotropy group of an arbitrary $\sigma$-derivation is obtained as the intersection of the isotropy groups of its components. Equivalently, it is described by character equations in the torus ${\Bbbk^\ast}^2$, and hence by arithmetic conditions in the lattice $\mathbb Z^2$. 

In Section \ref{classification-1-section} we study the singular case $q=-1$. We first characterize the $\sigma$-derivations of $\mathcal A_{-1}$, distinguishing between toric and flip automorphisms $\sigma$. We then use this classification to describe the corresponding isotropy groups. In the toric case, the toric part of
the isotropy is still governed by character equations, as in the case
$q\neq\pm1$, but the presence of flip automorphisms imposes additional compatibility conditions on the components of the derivation. Finally, when $\sigma=\mathrm{id}$, our results give an explicit description of the isotropy groups of ordinary derivations of $\mathcal A_{-1}$, extending the analysis in \cite{SBVA} to the case $q=-1$.

\section{The case $q^2 \neq 1$} \label{scharacterization}

Let $\Bbbk$ be an algebraically closed field of characteristic
zero, and let $\QPq = \Bbbk_q[x,y]$ be the quantum plane generated by $x$ and $y$ subject to the relation
\[
xy = qyx,
\]
where $q\in\Bbbk^\ast$. Throughout this section, we assume that $q\neq\pm1$. In this case, every automorphism of $\mathcal A_q$ is toric. Thus, for a fixed automorphism $\sigma$, we may write
\[
\sigma(x) = \alpha x, \quad \sigma(y) = \beta y,
\quad \alpha, \beta \in \Bbbk^\ast.
\]

We shall use the following notation for the center. If $q$ is a root of unity, let $t$ be its order in $\Bbbk^\ast$; if $q$ is not a root of unity, we take $t=0$. We denote by
\[
\centro(\QPq) = \Bbbk[x^t,y^t]
\]
the center of $\QPq$, with the convention that $\Bbbk[x^0,y^0]=\Bbbk$.

We reformulate Jordan's classification for the quantum plane
using our convention for $\sigma$-derivations. We then use
the resulting decomposition to describe their isotropy
groups. When $\sigma=\operatorname{id}$, this recovers
the description for ordinary derivations.

\subsection{A reformulation of Jordan's classification}

The classification of skew derivations of quantum affine spaces for toric automorphisms was obtained recently by D. Jordan \cite{Jordan2025}. In particular, his results contain the classification of $\sigma$-derivations of the quantum plane for toric automorphisms. We recall below this classification in a form adapted to the isotropy problem, after translating from the convention of left skew derivations to the convention used in the present paper (see, \cite[Proposition 5.3, Remark 5.5 and Examples 7.2]{Jordan2025}). More precisely, we separate the inner component and the non-inner families into independent $\aut(\mathcal A_q)$-submodules, allowing the isotropy group to be obtained as an intersection of the isotropy groups of these components.

\begin{theorem}\label{classification}
	Let $\sigma(x) = \alpha x$, $\sigma(y) = \beta y$ be an automorphism of the quantum plane $\mathcal A_q$, with $q^2 \neq 1$. Then, each $\sigma$-derivation on $\mathcal A_q$ can be decomposed into
	\[
    \delta = [w,\_]_{\sigma}+\mathcal{P}_x + \mathcal{P}_y + \lambda_1(x^t,y^t) D_x + \lambda_2(x^t,y^t) D_y 
    \]
	where $\lambda_1, \lambda_2$ belong to the center of $\mathcal A_q$ and
	\begin{itemize}
		\item $[w,\_]_{\sigma}$ is the inner $\sigma$-derivation induced by some $w \in \mathcal A_q$;
		
		\item if $\beta = q$ then $\mathcal{P}_x$ is the (non-inner) $\sigma$-derivation induced by
		\[
        \left\{ \begin{array}{l}
			\mathcal{P}_x(x) = a(y) \\
			\mathcal{P}_x(y) = 0
		\end{array} \right. \; , \quad \text{for some } a(y) \in \Bbbk[y],
        \]
		but if $\beta \neq q$ then $\mathcal{P}_x = 0$;
		
		\item if $\alpha = q^{-1}$ then $\mathcal{P}_y$ is the (non-inner) $\sigma$-derivation induced by
		\[
        \left\{ \begin{array}{l}
			\mathcal{P}_y(x) = 0 \\
			\mathcal{P}_y(y) = b(x)
		\end{array} \right. \; , \quad \text{for some } b(x) \in \Bbbk[x],
        \]
		but if $\alpha \neq q^{-1}$ then $\mathcal{P}_y = 0$;
		
		\item if $\alpha = q^{n}$ and $\beta = q^{-m}$ for some $n,m \geq 0$ then, for the smallest nonnegative integers $m,n$ for which this occurs, $D_x$ and $D_y$ are the (non-inner) $\sigma$-derivations induced by
		\[
        \left\{ \begin{array}{l}
			D_x(x) = (y^n x^m) x \\
			D_x(y) = 0
		\end{array} \right. \qquad \text{and} \qquad \left\{ \begin{array}{l}
			D_y(x) = 0 \\
			D_y(y) = (y^n x^m) y
		\end{array} \right. \; ,
        \]
		respectively, but if $\alpha \neq q^{n}$ for all $n \geq 0$ or $\beta \neq q^{-m}$ for all $m \geq 0$ then $D_x = D_y = 0.$
	\end{itemize}
\end{theorem}

In particular, every $\sigma$-derivation of $\mathcal A_q$ is inner in the following cases:
\begin{itemize}
	\item $\alpha \neq q^{-1}$ and $\beta \neq q^{-m}$ for all $m \geq -1$; or
	\item $\beta \neq q$ and $\alpha \neq q^n$ for all $n \geq -1$.
\end{itemize}

\begin{proof}
The result is a reformulation of Jordan's classification
for the quantum plane. To translate conventions, let
$\delta\in\der_\sigma(\QPq)$ and denote
$\partial=\delta\circ\sigma^{-1}$. Then, $\partial(ab)=\partial(a)b+\sigma^{-1}(a)\partial(b)$, $a,b\in\QPq$. Thus, $\partial$ is a left $\sigma^{-1}$-derivation in Jordan's convention. Conversely, every such
$\partial$ determines a $\sigma$-derivation
$\delta=\partial\circ\sigma$ in our convention. Moreover, $(I_w\circ\sigma^{-1})(a)=wa-\sigma^{-1}(a)w$, so this correspondence also preserves the inner
components. Applying \cite[Remarks 5.5(i) and Examples 7.2]{Jordan2025} with toric parameters $\alpha^{-1}$ and $\beta^{-1}$ and translating back to our convention, the component denoted by $F$ corresponds to the family
$\lambda_1D_x+\lambda_2D_y$, while the component
denoted by $E$ corresponds to the exceptional families
$\mathcal P_x+\mathcal P_y$. Together with the inner
component, this gives the stated decomposition.
\end{proof}

\subsection{Isotropy groups}

Let $A$ be a $\Bbbk$-algebra, $\sigma \in \aut(A)$ be a fixed automorphism of $A$ and $\delta\in \der_\sigma(A)$. As observed above, conjugation preserves $\der_\sigma(A)$ precisely for those $\rho \in \aut(A)$ that commute with $\sigma$. We denote so the \emph{centralizer}
\[
  \Cts := C_{\aut(A)}(\sigma) = \{ \rho \in \aut(A) \mid \rho \sigma = \sigma \rho \}.
\]

Then $\Cts$ acts on $\der_\sigma(A)$ by conjugation. For $ \delta \in \der_\sigma(A)$, we also define the \textit{isotropy} group of $\delta$:
\[
  \aut^{\Cts}_{\delta}(A) := \{ \rho \in \Cts \mid \rho \delta \rho^{-1} = \delta\}.
\]

By Alev and Chamarie \cite[Proposition 1.4.4]{AlevChamarie}, in the case of the quantum plane $\mathcal A_q$,  if $q \neq \pm 1$ then any automorphism $ \rho \in \aut(\mathcal A_q)$ is such that $ \rho(x) = \mu_1 x $ and $ \rho(y) = \mu_2 y $, with $ \mu_1, \mu_2 \in \Bbbk^* $. Thus, $\aut(\mathcal A_q) \cong {\Bbbk^*}^2 $ and therefore
\[
\Cts = C_{\aut(\mathcal A_q)}(\sigma)=\aut(\mathcal A_q).
\]

\begin{lemma}\label{lemaindep}
Let $G\subseteq \aut(A)$ be a subgroup acting on $\der_\sigma(A)$ by conjugation. Let $\mathcal D_1,\mathcal D_2\subseteq \der_\sigma(A)$ be $G$-stable subspaces such that $\mathcal D_1\cap\mathcal D_2=\{0\}$. Then, for any $\delta_1 \in \mathcal D_1$, $\delta_2 \in \mathcal D_2$ and $a,b \in \Bbbk^*$, we have
\[
\aut^{G}_{a \delta_1 + b \delta_2}(A)
=
\aut^{G}_{\delta_1}(A)\cap \aut^{G}_{\delta_2}(A).
\]
\end{lemma}

\begin{proof}
If $\rho \in \aut^{G}_{a \delta_1 + b \delta_2}(A)$ then
\[
  a \delta_1 + b \delta_2 = \rho (a \delta_1 + b \delta_2) \rho^{-1} = a(\rho \delta_1 \rho^{-1}) + b(\rho \delta_2 \rho^{-1})
  \in \mathcal D_1\oplus \mathcal D_2
\]
so
\[
  a(\rho \delta_1 \rho^{-1} - \delta_1) + b(\rho \delta_2 \rho^{-1} - \delta_2) = 0,
\]
with the first summand in $\mathcal D_1$ and the second in $\mathcal D_2$. Independence yields $\rho \delta_1 \rho^{-1} = \delta_1$ and $\rho \delta_2 \rho^{-1} = \delta_2$. Thus $\rho \in \aut^{G}_{\,\delta_1}(A) \cap \aut^{G}_{\,\delta_2}(A)$. The reverse inclusion is immediate.
\end{proof}

The Lemma \ref{lemaindep} indicates that direct sum decompositions of $\der_\sigma(A)$ are particularly useful in the study of isotropy groups. Indeed, when the summands are independent, the action of an automorphism on a
$\sigma$-derivation can be analyzed through its action on each component. The next result shows that $\der_\sigma(\QPq)$ may be obtained as a direct sum of one-dimensional $\aut(\QPq)$-submodules.  

\begin{proposition}\label{independent}
If $q\neq\pm1$, then $\der_\sigma(\QPq)$ is a direct sum of
one-dimensional $\aut(\QPq)$-submodules. More precisely,

	\begin{enumerate}

\item[(1)] the \textbf{inner family} is the direct sum
$\bigoplus\limits_{i,j\geq0}\Bbbk\,I_{y^jx^i}$.

        \item[(2)] (if $\beta = q$) the \textbf{$\mathcal P_x$-family} is the direct sum $\bigoplus\limits_{k \geq 0} \Bbbk \, \mathcal{X}_k$ of the submodules generated by the $\sigma$-derivations
		\[
		\left\{ \begin{array}{l}
			\mathcal{X}_k(x) = y^k, \\
			\mathcal{X}_k(y) = 0.
		\end{array} \right. \, 
		\]

        \item[(3)] (if $\alpha = q^{-1}$) the \textbf{$\mathcal P_y$-family} is the direct sum $\bigoplus\limits_{k \geq 0} \Bbbk \, \mathcal{Y}_k$ of the submodules generated by the $\sigma$-derivations
		\[
		\left\{ \begin{array}{l}
			\mathcal{Y}_k(x) = 0, \\
			\mathcal{Y}_k(y) = x^k.
		\end{array} \right. \, 
		\]

        \item[(4)] (if $\alpha = q^{n}$ and $\beta = q^{-m}$) the \textbf{families $\lambda_1 D_x$ and $\lambda_2 D_y$} are the direct sums $\bigoplus\limits_{r,s \geq 0} \Bbbk \, y^{st} x^{rt} D_x$ and $\bigoplus\limits_{r,s \geq 0} \Bbbk \, y^{st} x^{rt} D_y$, respectively, where the sums reduce to $\Bbbk D_x$ and $\Bbbk D_y$ when $q$ is not a root of unit.

    \end{enumerate}
\end{proposition}

\begin{proof}

By Theorem \ref{classification}, we obtain a decomposition of $\der_\sigma(\QPq)$ as a sum of these vector subspaces. It remains to verify that each of these summands is stable under conjugation by automorphisms of $\QPq$. Indeed, let $\rho(x)=\mu_1 x$ and $\rho(y)=\mu_2 y$ be an automorphism of $\QPq$.

If $w=\sum c_{ij}y^jx^i$, then $I_w = \sum c_{ij}I_{y^jx^i}$. Moreover, for every $i,j \geq 0$,
\begin{eqnarray} \label{actiontoricIji}
    \rho I_{y^jx^i} \rho^{-1} = I_{\rho(y^jx^i)} = I_{\mu_1^i\mu_2^j\,y^jx^i} = \mu_1^i\mu_2^j\,I_{y^jx^i} \; \in \; \Bbbk \, I_{y^jx^i},
\end{eqnarray}
so each such subspace $\Bbbk\,I_{y^jx^i}$ is a $\aut(\mathcal A_q)$-submodule.

If $\beta = q$, then it is easy to see that each $\mathcal{X}_k$ is a $\sigma$-derivation and each $\mathcal{P}_x$ belongs to $\bigoplus\limits_{k \geq 0} \Bbbk \, \mathcal{X}_k$. Moreover, for every $k \geq 0$,
\begin{eqnarray} \label{actionXk}
    \rho \mathcal{X}_k \rho^{-1} = \mu_1^{-1} \mu_2^k \mathcal{X}_k \; \in \; \Bbbk \, \mathcal{X}_k.
\end{eqnarray}
Therefore, each subspace $\Bbbk \, \mathcal{X}_k$ is an $\aut(\mathcal{A}_q)$-submodule. In (3) the argument is analogous to that in (2).

Finally, suppose that $\alpha = q^{n}$ and $\beta = q^{-m}$, with $m,n$ being the smallest nonnegative integers for which this occurs. If $q$ is a root of unity of order $t$, then the center of $\QPq$ is $\Bbbk[x^t,y^t]$ so $\lambda_1 D_x$ belongs to $\bigoplus\limits_{r,s \geq 0} \Bbbk \, y^{st} x^{rt} D_x$ and $\lambda_2 D_y$ belongs to $\bigoplus\limits_{r,s \geq 0} \Bbbk \, y^{st} x^{rt} D_y$. Since for every $r,s \geq 0$, we have
        \begin{eqnarray} \label{actionDxDy}
            \left\{ \begin{array}{l}
            \rho (y^{st} x^{rt} D_x) \rho^{-1} = \mu_1^{m+rt} \mu_2^{n+st} \, y^{st} x^{rt} D_x \; \in \; \Bbbk \, y^{st} x^{rt} D_x \\[0.2cm]
            \rho (y^{st} x^{rt} D_y) \rho^{-1} = \mu_1^{m+rt} \mu_2^{n+st} \, y^{st} x^{rt} D_y \; \in \; \Bbbk \, y^{st} x^{rt} D_y
            \end{array} \right. \, ,
        \end{eqnarray}
        each $\Bbbk \, y^{st} x^{rt} D_x$ and each $\Bbbk \, y^{st} x^{rt} D_y$ is an $\aut(\mathcal{A}_q)$-submodule.
        
        If $q$ is not a root of unit, then $t = 0$, (\ref{actionDxDy}) becomes $\rho D_x \rho^{-1} = \mu_1^{m} \mu_2^{n} \, D_x$ and $\rho D_y \rho^{-1} = \mu_1^{m} \mu_2^{n} \, D_y$ and the sums reduce to $\Bbbk D_x$ and $\Bbbk D_y$.

\end{proof}

We now study the isotropy groups of the families that appear in the classification of $\sigma$-derivations of $\QPq$ (Theorem \ref{classification}). For an automorphism  $\rho(x)=\mu_1 x$ and $\rho(y)=\mu_2 y$ of $\QPq$, we determine, for each family of $\sigma$-derivations, the conditions on $\mu_1$ and $\mu_2$ under which $\rho \delta \rho^{-1} = \delta$. Before, note that for the inner family, for each $i,j$, we have
\[
\left\{ \begin{array}{l}
	I_{y^j x^i}(x) = (\alpha - q^j) y^j x^i x \\[0.2cm]
	I_{y^j x^i}(y) = (\beta - q^{-i}) y^j x^i y
\end{array} \right.
\]
so
\begin{eqnarray} \label{innerzero}
    I_{y^j x^i} = 0 \quad \Longleftrightarrow \quad \alpha = q^j \; \text{ and } \; \beta = q^{-i} \; .
\end{eqnarray}

\begin{proposition} \label{isotropy}
The isotropy groups of the families that appear in the classification of
$\sigma$-derivations are given as follows:

\begin{enumerate}

\item[(1)] The \textbf{inner family.} Let $I_{w} = [w, \_]_{\sigma}$ be the inner $\sigma$-derivation induced by $w = \sum c_{ij} y^j x^i$. Then,
\[
\aut_{I_{w}}(\mathcal A_q)
=
\left\{ (\mu_1,\mu_2) \in {\Bbbk^\ast}^2  ; \, \mu_1^i\mu_2^j=1, \, \forall \, (i,j) \text{ s.t. }c_{ij} I_{y^j x^i} \neq 0 \right\}.
\]

\item[(2)] The \textbf{$\mathcal P_x$-family.} Suppose $\beta = q$ and let $\mathcal P_x$ be defined by
$\mathcal P_x(x) = \sum a_k y^k$, $\mathcal P_x(y)=0$. Then,
\[
\aut_{\mathcal P_x}(\mathcal A_q)
=
\left\{ (\mu_1,\mu_2) \in {\Bbbk^\ast}^2  ; \,
\mu_1 = \mu_2^k,\  \forall \, k \text{ s.t. } a_k \neq 0 \right\}.
\]

\item[(3)] The \textbf{$\mathcal P_y$-family.} Suppose $\alpha = q^{-1}$ and let  $\mathcal P_y$ be defined by
$\mathcal P_y(x)=0$, $\mathcal P_y(y) = \sum b_k x^k$. Then,
\[
\aut_{\mathcal P_y}(\mathcal A_q)
=
\left\{ (\mu_1,\mu_2) \in {\Bbbk^\ast}^2  ; \,
\mu_2 = \mu_1^k,\ \forall \, k \text{ s.t. }  b_k \neq 0 \right\}.
\]

\item[(4)] The \textbf{families $\lambda_1 D_x$ and $\lambda_2 D_y$.} Suppose that $\alpha = q^{n}$ and $\beta = q^{-m}$, with $m,n$ being the smallest nonnegative integers for which this occurs. Then, for a central element $\lambda = \sum d_{rs} y^{st} x^{rt} \in \Bbbk[x^t,y^t]$,
\[
\aut_{\lambda D_x}(\mathcal A_q) = \left\{ (\mu_1,\mu_2) \in {\Bbbk^\ast}^2  ; \,
\mu_1^{m+rt}\mu_2^{n+st}=1, \, \forall \, (r,s) \text{ s.t. } d_{rs} \neq 0 \right\}.
\]
and,
\[
\aut_{\lambda D_y}(\mathcal A_q) = \left\{ (\mu_1,\mu_2) \in {\Bbbk^\ast}^2  ; \,
\mu_1^{m+rt}\mu_2^{n+st}=1, \, \forall \, (r,s) \text{ s.t. } d_{rs} \neq 0 \right\}.
\]

\end{enumerate}
\end{proposition}

\begin{proof}
Let $\rho(x)=\mu_1 x$ and $\rho(y)=\mu_2 y$ be an automorphism of $\mathcal A_q$. If $I_{y^j x^i} \neq 0$ then, by equation (\ref{actiontoricIji}), $\rho \in \aut_{I_{y^j x^i}}(\QPq)$ if and only if $\mu_1^i \mu_2^j = 1$. Since for $w = \sum c_{ij} y^j x^i$ we have $I_w = \sum c_{ij} I_{y^j x^i}$, item (1) then follows from Lemma \ref{lemaindep} and Proposition \ref{independent}. 

Items (2), (3) and (4) are analogous and follow from equations (\ref{actionXk}) and (\ref{actionDxDy}).

\end{proof}

Also, due to Lemma \ref{lemaindep} and Proposition \ref{independent}, the isotropy of an arbitrary $\sigma$-derivation
\[
\delta = [w,\_]_{\sigma}+\mathcal{P}_x + \mathcal{P}_y + \lambda_1(x^t,y^t) D_x + \lambda_2(x^t,y^t) D_y 
\]
is obtained by taking the intersection of the isotropy groups of the individual parts. This intersection can be described in terms of characters of the torus ${\Bbbk^\ast}^2$. For $(u,v)\in\mathbb Z^2$, denote by
\[
\chi_{u,v}:{\Bbbk^\ast}^2\to \Bbbk^\ast
\]
the character $\chi_{u,v}(\mu_1,\mu_2)=\mu_1^u\mu_2^v$ (see \cite[Chapter III]{Borel}): therefore, each family contributes with equations of the form $\chi_{u,v}(\mu_1,\mu_2)=1$. 

More explicitly, for $[w,\_]_\sigma$ with $w=\sum c_{ij}y^j x^i$, the inner part contributes with $(i,j) \in\mathbb Z^2$ such that $c_{ij} I_{y^jx^i} \neq 0$. If $\mathcal P_x(x) = \sum a_k y^k$ and $\mathcal P_x(y)=0$, then the family $\mathcal P_x$ contributes with $(1,-k) \in\mathbb Z^2$ whenever $a_k \neq 0$, because the isotropy condition $\mu_1 = \mu_2^k$. Similarly, if $\mathcal P_y(x) = 0$ and $\mathcal P_y(y) = \sum b_k x^k$ then this family contributes with $(-k,1) \in \mathbb Z^2$ whenever $b_k \neq 0$. Finally, if
\[
\lambda_1(x^t,y^t) = \sum_{r,s \geq 0} d_{rs}^{(1)} y^{st} x^{rt}
\quad \text{ and } \quad
\lambda_2(x^t,y^t) = \sum_{r,s \geq 0} d_{rs}^{(2)} y^{st} x^{rt},
\]
then the terms $\lambda_1 D_x$ and $\lambda_2 D_y$ contribute with $(m+rt,n+st) \in \mathbb Z^2$ whenever $d_{rs}^{(1)} \neq 0$ or $d_{rs}^{(2)} \neq 0$.

Therefore, if $\Gamma(\delta)\subseteq\mathbb Z^2$ denotes the finite set of all elements obtained from the nonzero summands of $\delta$, then
\begin{eqnarray} \label{eqcaracter}
    \aut_\delta(\QPq) & = & \left\{ (\mu_1,\mu_2)\in {\Bbbk^\ast}^2 ; \, \mu_1^u \mu_2^v = 1, \, \forall \, (u,v) \in \Gamma(\delta) \right\} \nonumber \\
    & = & \bigcap_{(u,v) \in \Gamma(\delta)} \ker \chi_{u,v}
\end{eqnarray}
and the description of $\aut_\delta(\mathcal A_q)$ reduces to an arithmetic problem in the lattice $\mathbb Z^2$.

The next remark provides natural analogues, in the setting of $\sigma$-derivations, of the corresponding results on finiteness criterion for isotropy groups obtained in \cite[Proposition 5]{SBVA}.

\begin{remark}
Let $0\neq\delta\in\der_\sigma(\mathcal A_q)$. If $\mathcal P_x=\mathcal P_y=0$, all exponent vectors
have nonnegative coordinates. In this case, the arithmetic
criterion in \cite[Proposition 5]{SBVA} implies that the
isotropy group is finite if and only if there exist
$(u_i,v_i),(u_j,v_j)\in\Gamma(\delta)$ such that
$u_iv_j-u_jv_i\neq0$.

The new phenomenon in the $\sigma$-derivation setting is
the possible appearance of the families $\mathcal P_x$
and $\mathcal P_y$. These families contribute with equations
of the forms
\[
\mathcal P_x:\ \mu_1=\mu_2^r,
\quad
\mathcal P_y:\ \mu_2=\mu_1^s,
\]
for $a_r\neq0$ and $b_s\neq0$, respectively.
These equations correspond to the exponent vectors
$(1,-r)$ and $(-s,1)$. The same finiteness criterion holds in the presence of these exceptional families. More precisely, the isotropy group is finite if and only if there exist
$(u_i,v_i),(u_j,v_j)\in\Gamma(\delta)$ such that $u_iv_j-u_jv_i\neq0$. Indeed, suppose that $\mu_1^{u_i}\mu_2^{v_i}=1$ and $\mu_1^{u_j}\mu_2^{v_j}=1$. Writing $d=u_iv_j-u_jv_i$, we obtain $\mu_1^d=1$ and $\mu_2^d=1$. If $d\neq0$, both $\mu_1$ and $\mu_2$ have only finitely many possible values, so the isotropy group is finite.

Conversely, suppose that all pairwise determinants of vectors in $\Gamma(\delta)$ vanish. If
$\Gamma(\delta) = \{ (0,0) \}$, then $\aut_\delta(\QPq)={\Bbbk^\ast}^2$, which is infinite.
Otherwise, fix $ (u,v)\in\Gamma(\delta) \setminus\{(0,0)\}$. Then, we have 
\[
\left\{ (\mu^{v},\mu^{-u}) \; \middle| \; \forall \, \mu \in \Bbbk^{\ast} \right\} \subseteq \aut_\delta(\QPq),
\]
and therefore, $\aut_\delta(\QPq)$ is infinite.

Consequently, the finiteness problem for
$\sigma$-derivations is governed by the same arithmetic
criterion as in \cite[Proposition 5]{SBVA}, now allowing
integer exponent vectors rather than only nonnegative
ones.
\end{remark}

\subsection{Inner $\sigma$-derivation}

As before, denote by $n$ and $m$ the smallest nonnegative integers such that $\alpha = q^{n}$ and $\beta = q^{-m}$, if such integers exist. In this case, let $t$ be the order of $q$ in $\Bbbk^\ast$ if $q$ is a root of unity, but $t = 0$ otherwise. For $w = \sum c_{ij} y^j x^i$, consider the decomposition
\begin{eqnarray} \label{wdecomp}
    w = w' + \Bar{w}
\end{eqnarray}
where
\[
\Bar{w} = \sum c_{m+rt, n+st} y^{n+st} x^{m+rt}
\]
is the sum of the terms of $w$ in $\Bbbk[x^t,y^t] \, y^n x^m$. When $t=0$, the sum that defines $\Bar{w}$ consists only of the term $c_{mn}y^nx^m$. Moreover, if $\alpha \neq q^{n}$ for all $n \geq 0$ or $\beta \neq q^{-m}$ for all $m \geq 0$, then $\Bar{w} = 0$ and $w = w'$ for all $w \in \QPq$.

By (\ref{innerzero}) we have $I_{\Bar{w}} = 0$ and so $I_w = I_{w'}$. In other words, when dealing with inner $\sigma$-derivation we can ignore the terms in
$\Bbbk[x^t,y^t]\cdot y^n x^m$ and consider only the others.
Moreover,
\begin{equation}\label{innereq}
    I_w = I_{w'} = 0 \quad \iff \quad w' = 0.
\end{equation}

The next result shows that the isotropy group of an inner $\sigma$-derivation $I_w$ is determined by the condition $\rho(w')=w'$, where $w'$ is obtained from $w$ by removing the terms that induce the zero $\sigma$-derivation.

\begin{proposition}\label{pho(x)}
Let $I_w$ be the inner $\sigma$-derivation induced by $w\in \mathcal A_q$. Taking into account the decomposition $w = w' + \Bar{w}$ described above, we have
\[
\rho \in \aut_{I_w}(\QPq)
\iff
\rho(w') = w'.
\]
In particular, if $\alpha \neq q^{n}$ for all $n \geq 0$ or $\beta \neq q^{-m}$ for all $m \geq 0$, then
\[
\rho \in \aut_{I_w}(\QPq)
\iff
\rho(w) = w
\]
for any $w \in \QPq$.
\end{proposition}

\begin{proof}
    Since $\rho I_{w} \rho^{-1} = I_{\rho(w)}$ it follows that
    \[
    \rho \in \aut_{I_w}(\mathcal A_q)
    \iff
    I_{\rho(w)} = I_w
    \iff
    I_{\rho(w) - w} = 0.
    \]
    On the other hand, since any automorphism of $\mathcal A_q$ is of the form $\rho(x)=\mu_1x$, $\rho(y)=\mu_2y$, with $\mu_1, \mu_2 \in \Bbbk^\ast$, it is clear that (following the notation above)
    \[
    \big( \rho(w) - w \big)' = \rho(w)' - w' = \rho(w') - w'
    \]
    so, by (\ref{innereq}),
    \[
    I_{\rho(w) - w} = 0
    \iff
    \rho(w') - w' = 0
    \iff
    \rho(w') = w'.
    \]
\end{proof}

The Proposition \ref{pho(x)} can be interpreted in terms of the \emph{twisted center} associated with $\sigma$:
\[
\centroT(\QPq) := \{f \in \mathcal A_q \mid gf = f \sigma(g),\ \text{for every } g \in \QPq\}.
\]
We have $f \in \centroT(\QPq)$ if and only if $[f,\_]_\sigma = 0$. So, by (\ref{innerzero}) and (\ref{innereq}),
\[
\centroT(\QPq) = \Bbbk[x^t,y^t]\,y^n x^m \; ,
\]
with $\centroT(\QPq) = 0$ if there are no integers $m,n \geq 0$ such that
$\alpha = q^n$ and $\beta = q^{-m}$. Thus, Proposition \ref{pho(x)} can be rewritten as follows.
\begin{proposition} \label{innercentersima}
    $\rho \in \aut_{I_w}(\QPq) \iff \rho(w) - w \in \centroT(\QPq)$.
\end{proposition}

Let $G$ be a finite subgroup of ${\Bbbk^*}^2$. Then, by the structure theorem for finite abelian groups,
\[
G \cong \mathbb{Z}_{n_1} \oplus \mathbb{Z}_{n_2} \oplus \ldots \oplus \mathbb{Z}_{n_r},
\]
where $n_{i+1} \mid n_i$. If $n_3 =: n \neq 1$, then there exists a subgroup $H$ of $G$ such that $H \cong \mathbb{Z}_n \oplus \mathbb{Z}_n \oplus \mathbb{Z}_n$. Note that the equation $x^n - 1 = 0$ has $n^3$ solutions in $H$, but at most $n^2$ solutions in ${\Bbbk^*}^2$. Therefore, $G$ is not a subgroup of ${\Bbbk^*}^2$. Thus, we conclude that $G \cong \mathbb{Z}_{n_1} \oplus \mathbb{Z}_{n_2}$, with $n_2 \mid n_1$.
\smallbreak
Note also that, since $\Bbbk$ is algebraically closed, for any positive integers $n_1,n_2$, there exists a subgroup $G < {\Bbbk^*}^2$ such that $G \cong \mathbb{Z}_{n_1} \oplus \mathbb{Z}_{n_2}$: it suffices to take $G = \langle (\xi_{n_1},1), (1,\xi_{n_2}) \rangle$, where $\xi_{n_i}$ is a primitive $n_i$-th root of unity.

\begin{proposition} \label{realization}
The class of isotropy groups of $\sigma$-derivations of $\mathcal A_q$ includes, up to isomorphism, all finite subgroups of $\aut(\mathcal A_q)$. 

\end{proposition}

\begin{proof}

Let $G$ be a finite subgroup of ${\Bbbk^*}^2$. As previously observed, there exist positive integers $r,s$ such that $G \cong \mathbb Z_r \oplus \mathbb Z_s$. We shall construct a $\sigma$-derivation $D$ of $\QPq$ such that
\[
\aut_D(\QPq) \cong G.
\]

If $\alpha \neq 1$ or $\beta \neq q^{-r}$ take $D_1 = I_{x^r}$ the inner $\sigma$-derivation induced by $x^r$. But if $\alpha = 1$ and $\beta = q^{-r}$ take the $\sigma$-derivation
\[
\left\{ \begin{array}{l}
     D_1(x) = x^{r+1} \\
     D_1(y) = 0
\end{array} \right. \; .
\]
Analogously, if $\alpha \neq q^s$ or $\beta \neq 1$ take $D_2 = I_{y^s}$ the inner $\sigma$-derivation induced by $y^s$. But if $\alpha = q^s$ and $\beta = 1$ take the $\sigma$-derivation
\[
\left\{ \begin{array}{l}
     D_2(x) = 0 \\
     D_2(y) = y^{s+1}
\end{array} \right. \; .
\]

For an automorphism $\rho(x) = \mu_1 x$, $\rho(y) = \mu_2 y$, it is easy to see that
\[
\rho \in \aut_{D_1}(\QPq)
\iff
\mu_1^r = 1
\]
and
\[
\rho \in \aut_{D_2}(\QPq)
\iff
\mu_2^s = 1 .
\]
By Proposition \ref{independent}, taking $D = D_1 + D_2$ we have
\[
\aut_{D}(\QPq) \cong \mathbb Z_r \oplus \mathbb Z_s.
\]
\end{proof}

\begin{remark}
The preceding result also clarifies a point concerning the ordinary root-of-unity case considered in  \cite[Proposition 9]{SBVA}. In that work, a non-realization
statement for certain finite abelian groups was obtained in the case $\sigma = \mathrm{id}$. Revisiting the proof, one sees that the argument uses only the equations arising from the inner part of an ordinary derivation. More precisely, the proof treats the isotropy as being determined by the commutation condition $\rho I_w = I_w \rho$. That is true for inner derivations, and also when the coefficients of the non-inner summands $D_x$ and $D_y$ are scalars. However, when $q$ has finite order $t$, the center
of $\QPq$ is $\Bbbk[x^t,y^t]$. Thus, non-inner summands with central coefficients may also contribute to the isotropy. Indeed, for positive integers $r,s$, the (ordinary) derivation
\[
D=x^{tr}D_x+y^{ts}D_y
\]
satisfies
\[
\aut_D(\mathcal A_q)
\cong
\mathbb Z_{tr}\oplus\mathbb Z_{ts}.
\]
\medbreak

Consequently, the non-realization statement in \cite{SBVA} should be read with an additional restriction to the inner ordinary case, or more generally to the case where the non-inner central coefficients do not affect the isotropy. Without such a restriction, the statement does not hold for all ordinary derivations.
\end{remark}

\begin{remark}
The previous description, Equation (\ref{eqcaracter}), also shows that one should not expect a realization result for arbitrary infinite subgroups of ${\Bbbk^\ast}^2$ in the
$\sigma$-derivation setting. Indeed, for example, if $\theta\in\Bbbk^\ast$ has infinite order, then
the cyclic subgroup
\[
\langle(\theta,1)\rangle
\subseteq {\Bbbk^\ast}^2
\]
is not Zariski closed in the torus. Thus,
$\langle(\theta,1)\rangle$ cannot be equal to $\aut_\delta(\mathcal A_q)$ for any
$\sigma$-derivation $\delta$. Therefore, as in the ordinary derivation case \cite[Remark 10]{SBVA}, for any automorphism $\sigma$, the class of isotropy groups of $\sigma$-derivations does not contain all infinite subgroups of ${\Bbbk^\ast}^2$.
\end{remark}

\section{The singular case \texorpdfstring{$q=-1$}{q=-1}}\label{classification-1-section}

In this section, we study $\sigma$-derivations of the quantum plane
$\mathcal A_{-1}$. Recall that $\centro(\mathcal A_{-1})=\Bbbk[x^2,y^2]$. Moreover, every automorphism of $\mathcal A_{-1}$ is of one of the following two types (see, \cite[Proposition 1.4.4]{AlevChamarie}). The first type is given by 
\[
\rho(x)=\mu_1x, \quad \rho(y)=\mu_2y, \quad \mu_1, \mu_2 \in \Bbbk^\ast .
\]
These automorphisms are naturally identified with the algebraic torus ${\Bbbk^\ast}^2$ in $\aut(\mathcal A_{-1})$ and will be called \emph{toric automorphisms}. The second type is given by
\[
\eta(x) = \nu_1 y, \quad \eta(y) = \nu_2 x, \quad \nu_1, \nu_2 \in \Bbbk^\ast.
\]
These automorphisms interchange the generators $x$ and $y$, up to nonzero scalar multiples, and will be called \emph{flip automorphisms}. Thus,
\[
\aut(\mathcal A_{-1})
\cong
(\Bbbk^\ast \times \Bbbk^\ast) \rtimes \langle \tau \rangle,
\]
where $\tau(x) = y$ and $\tau(y) = x$.

Let $\sigma \in \aut(\mathcal A_{-1})$ be fixed. As in the general setting of $\sigma$-derivations, conjugation by an automorphism $\rho \in \aut(\mathcal A_{-1})$ preserves $\der_\sigma(\mathcal A_{-1})$ precisely when $\rho$ commutes with $\sigma$. Therefore, the natural group acting on $\der_\sigma(\mathcal A_{-1})$ is the centralizer
\[
\Cts := C_{\aut(\QPs)}(\sigma)
=
\{ \rho \in \aut(\mathcal A_{-1}) \mid \rho \sigma = \sigma \rho \}.
\]

Thus, using the notation introduced in Section \ref{scharacterization}, for $\delta\in\der_\sigma(\QPs)$, we consider the isotropy group
\[
\aut_\delta(\QPs) := \aut^{\Cts}_\delta(\mathcal A_{-1})
=
\{ \rho \in \Cts \mid \rho \delta = \delta \rho \}.
\]

\begin{remark}\label{criterionsigma}
Let $\sigma \in \aut(\mathcal A_q)$. Notice that a $\sigma$-derivation is determined by its values on $x$ and $y$, provided that these values are compatible with the defining relation $xy=qyx$. This criterion is valid for any $q\in\Bbbk^\ast$. Thus, if a map defined on the generators by $\partial(x) = u$, $\partial(y) = v$, with $u,v\in \mathcal A_q$, satisfies
\[
u \sigma(y) + xv
=
q (v \sigma(x) + yu),
\]
then $\partial$ extends to a $\sigma$-derivation of $\mathcal A_q$.
\end{remark}

The next theorem gives a characterization of the $\sigma$-derivations of $\mathcal A_{-1}$, distinguishing between the toric and flip cases. The toric case follows from Jordan's  results. The flip case, however, is not treated in \cite{Jordan2025}.

\begin{theorem} \label{classification-1}
	Let $\sigma$ be an automorphism of the quantum plane $\mathcal A_{-1}$.
	\begin{itemize}
		\item[(1)] If $\sigma(x)=\alpha x$, $\sigma(y)=\beta y$ is toric, then the same conclusion as in Theorem \ref{classification} holds. That is, a $\sigma$-derivation $\delta$ on $\mathcal A_{-1}$ can be decomposed in
		$$\delta = \mathcal{P}_x + \mathcal{P}_y + \lambda_1(x^2,y^2) D_x + \lambda_2(x^2,y^2) D_y + [w,\_]_{\sigma}$$
		where $\lambda_1, \lambda_2$ belong to the center of $\mathcal A_{-1}$ and
		\begin{itemize}
			\item[$\bullet$] $[w,\_]_{\sigma}$ is the inner $\sigma$-derivation induced by some $w \in \mathcal A_{-1}$;
			
			\item[$\bullet$] if $\beta = -1$ then $\mathcal{P}_x$ is the (non-inner) $\sigma$-derivation induced by
			\[
            \left\{ \begin{array}{l}
				\mathcal{P}_x(x) = a(y) \\
				\mathcal{P}_x(y) = 0
			\end{array} \right. \; , \quad \text{for some } a(y) \in \Bbbk[y],
            \]
			but if $\beta \neq -1$ then $\mathcal{P}_x = 0$;
			
			\item[$\bullet$] if $\alpha = -1$ then $\mathcal{P}_y$ is the (non-inner) $\sigma$-derivation induced by
			\[
            \left\{ \begin{array}{l}
				\mathcal{P}_y(x) = 0 \\
				\mathcal{P}_y(y) = b(x)
			\end{array} \right. \; , \quad \text{for some } b(x) \in \Bbbk[x],
            \]
			but if $\alpha \neq -1$ then $\mathcal{P}_y = 0$;
			
			\item[$\bullet$] if $\alpha = \pm 1$ and $\beta = \pm 1$ then $D_x$ and $D_y$ are the (non-inner) $\sigma$-derivations induced by
			\[
            \left\{ \begin{array}{l}
				D_x(x) = (y^n x^m) x \\
				D_x(y) = 0
			\end{array} \right. \quad \text{and} \quad \left\{ \begin{array}{l}
				D_y(x) = 0 \\
				D_y(y) = (y^n x^m) y
			\end{array} \right.
            \]
			where
\[
n=
\begin{cases}
0, & \text{if } \alpha = 1 \\
1, & \text{if }\alpha = -1
\end{cases}
\qquad\text{and}\qquad
m=
\begin{cases}
0, & \text{if }\beta=1 \\
1, & \text{if }\beta=-1
\end{cases} ,
\]
			but if $\alpha \neq \pm 1$ or $\beta \neq \pm 1$ then $D_x = D_y = 0.$
		\end{itemize}
		
		\item[(2)] If $\sigma(x)=\alpha y$, $\sigma(y)=\beta x$ is flip,
		then, for a $\sigma$-derivation $\delta$,
		\begin{itemize}
			\item[$\bullet$] if $\alpha \beta \neq \pm 1$ then $\delta$ is an inner $\sigma$-derivation;
			
			\item[$\bullet$] if $\alpha \beta = \pm 1$ then
            \[
            \delta = [w,\_]_{\sigma} + \sum\limits_{\alpha \beta = (-1)^k} \varepsilon_k \dk
            \]
			for some $w \in \QPs$, and where each $\varepsilon_k \in \Bbbk$ and $\dk$ is the (non inner) $\sigma$-derivation given by
			\begin{eqnarray} \label{derdk}
				\left\{ \begin{array}{l}
					\dk(x) = - \beta^{-1} \sum\limits_{l=0}^{k} (-1)^l \beta^{k-l} y^l x^{k-l} \\ [0.3cm]
					\dk(y) = \sum\limits_{l=0}^{k} \beta^{k-l} y^l x^{k-l}
				\end{array} \right. .
			\end{eqnarray}
		\end{itemize}
	\end{itemize}
\end{theorem}
\begin{proof}
The proof of $(1)$ is similar to Theorem \ref{classification}. For $(2)$, let $\sigma(x) = \alpha y$, $\sigma(y) = \beta x$ and write the homogeneous decomposition of $\delta = \sum_{k \geq 0} \delta_k$, where each
\begin{eqnarray} \label{sdergeral}
	\left\{ \begin{array}{l}
		\delta_k(x) = a_0 x^k + a_1 y x^{k-1} + a_2 y^2 x^{k-2} + \ldots + a_k y^k \\ [0.5cm]
		\delta_k(y) = b_0 x^k + b_1 y x^{k-1} + b_2 y^2 x^{k-2} + \ldots + b_k y^k
	\end{array} \right.
\end{eqnarray}
correspond to the monomials of degree $k$ when $\delta$ is applied in $x$ and $y$.

By Remark \ref{criterionsigma}, $\delta$ is a $\sigma$-derivation if and only if
\[
\delta(x) \sigma(y) + x \delta(y)
=
- \left( \delta(y) \sigma(x) + y \delta(x) \right).
\]
Comparing the homogeneous components of total degree $k+1$ on both sides of this equality, we obtain
\[
\delta_k(x)\sigma(y)+x\delta_k(y)
=
- \left( \delta_k(y)\sigma(x)+y\delta_k(x) \right).
\]
Thus, each $\delta_k$ is a $\sigma$-derivation. Moreover, since
\[
\delta_k(x) \sigma(y) + x \delta_k(y)
=
\sum_{l=0}^{k}
\left( a_l \beta + (-1)^l b_l \right) y^l x^{k+1-l},
\]
and
\[
- (\delta_k(y) \sigma(x) + y \delta_k(x))
=
\sum_{l=1}^{k+1}
\left( (-1)^{k-l} b_{l-1} \alpha - a_{l-1} \right)
y^l x^{k+1-l}.
\]
the coefficients of (\ref{sdergeral}) satisfy
\begin{eqnarray} \label{condDeltaK}
		\left\{ \begin{array}{l}
			a_0 \beta + b_0 = 0 \\ [0.2cm]
			a_l \beta + (-1)^l b_l = (-1)^{k-l} b_{l-1} \alpha - a_{l-1} \; , \quad 1 \leq l \leq k \\ [0.2cm]
			a_k + b_k \alpha = 0
		\end{array} \right.
	\end{eqnarray}
In particular, for $k = 0$, if $\alpha \beta \neq 1$ then $\delta_0 = 0$.

For $k \geq 1$, if $\alpha \beta \neq (-1)^k$, then, for each arbitrary choice $b_1, b_2, \ldots, b_k \in \Bbbk$, (\ref{condDeltaK}) has a unique solution given by
\begin{eqnarray*}
	\left\{ \begin{array}{l}
    	b_0 \; = \; - \big( b_1 \beta + b_{2} \beta^2 + b_{3} \beta^{3} + \ldots + b_{k} \beta^{k} \big) \\ [0.3cm]
        a_0 \; = \; b_1 + b_{2} \beta + b_{3} \beta^{2} + \ldots + b_{k} \beta^{k-1} \\ [0.3cm]
		a_{l} \; = \; (-1)^{k-l-1} \left( b_{l} \alpha + \left( \alpha \beta - (-1)^k \right) \sum\limits_{i=l+1}^{k} b_{i} \beta^{i-l-1} \right) \, , \;\; 1 \leq l \leq k-1 \\ [0.3cm]
		a_k \; = \; - b_k \alpha
	\end{array} \right. \, .
\end{eqnarray*}
In this case, it is easy to verify that $\delta_k=[w_k,\_]_{\sigma}$ is the inner $\sigma$-derivation induced by
\begin{eqnarray}\label{wkinner}
w_k = \epsilon_1 x^{k-1} + \epsilon_2 y x^{k-2}
+ \epsilon_3 y^2 x^{k-3} + \ldots + \epsilon_k y^{k-1},
\end{eqnarray}
with $\epsilon_l=-\left(b_l+b_{l+1}\beta+\ldots
+b_k\beta^{k-l}\right)$ for each $1\leq l\leq k$.

On the other hand, for $k \geq 0$, if $\alpha\beta=(-1)^k$, then, for each choice 
\[
b_0, b_1, b_2, \ldots, b_k \in \Bbbk,
\]
(\ref{condDeltaK}) has a unique solution given by $a_l = (-1)^{l-1} \beta^{-1} b_l$, for $l = 0, \ldots, k$. That is,
\begin{eqnarray} \label{deltak}
	\left\{ \begin{array}{l}
		\delta_k(x) = - \beta^{-1} \left( b_0 x^k - b_1 y x^{k-1} + b_2 y^2 x^{k-2} + \ldots + (-1)^k b_k y^k \right) \\ [0.3cm]
		\delta_k(y) = b_0 x^k + b_1 y x^{k-1} + b_2 y^2 x^{k-2} + \ldots + b_k y^k
	\end{array} \right.
\end{eqnarray}
In this case, $\delta_k$ is an inner $\sigma$-derivation, induced by an element $w_k$ of the form given in (\ref{wkinner}), if and only if
\begin{eqnarray} \label{bl}
    b_0 = - \left( b_1 \beta + b_2 \beta^2 + \ldots + b_k \beta^k \right).
\end{eqnarray}

In particular, when $\alpha\beta=(-1)^k$, the subspace
$\Delta_k$ of all $\sigma$-derivations as in
(\ref{deltak}) has dimension $k+1$. Since $\dk$
satisfies (\ref{deltak}) but not (\ref{bl}), $\dk$
is a non-inner $\sigma$-derivation, and so $\Delta_k = \mathcal I_k + \Bbbk \dk$, where $\mathcal I_k$ is the $k$-dimensional subspace of all inner $\sigma$-derivations induced by homogeneous elements of degree $k-1$. This proves item $(2)$.
\end{proof}

\vspace{0.2cm}
Note that when $\alpha \beta = (-1)^k$, as a vector space, $\Delta_k$ admits other
decompositions that may be more convenient depending
on the context. For example, for each $0 \leq l \leq k$,
\begin{eqnarray*}
	\left\{ \begin{array}{l}
		\dfk_l(x) = (-1)^{l+1} \beta^{-1} y^l x^{k-l} = (-1)^{k-l+1} \alpha \, y^l x^{k-l} \\ [0.5cm]
		\dfk_l(y) = y^l x^{k-l}
	\end{array} \right.
\end{eqnarray*}
satisfy (\ref{deltak}) and we have $\Delta_k = \sum_{l=0}^{k} \Bbbk \dfk_l$. Or, since none of the $\sigma$-derivations $\dfk_l$ satisfies (\ref{bl}), we also have $\Delta_k = \mathcal{I}_k + \Bbbk \dfk_l$ for any $l$. We choose $\dk$ because the subspace $\Bbbk\,\dk$ is stable under conjugation by the centralizer $\Cts$ (see Proposition \ref{independent3}).
\smallbreak

Using the classification obtained above, we study, as in
Section \ref{scharacterization}, the isotropy groups of $\sigma$-derivations in the singular case $q=-1$. Since flip automorphisms may occur, we begin by describing the centralizer of $\sigma$ in $\aut(\mathcal A_{-1})$.

\begin{lemma}\label{centralizador}
Let $\sigma \in \aut(\QPs)$.

\begin{enumerate}
\item[(i)] Suppose that $\sigma(x) = \alpha x$, $\sigma(y)=\beta y$ is toric. Then, 
\[
\Cts := C_{\aut(\QPs)}(\sigma)
=
\begin{cases}
 {\Bbbk^\ast}^2, & \text{if } \alpha \neq \beta, \\[0.15cm]
\aut(\QPs), & \text{if } \alpha = \beta.
\end{cases}
\]

\item[(ii)] Suppose that $\sigma(x)=\alpha y$, $\sigma(y)=\beta x$ is flip. Then,
\[
\Cts := C_{\aut(\QPs)}(\sigma) = \Ct_0 \cup \Ct_1, 
\]
where
\[
\Ct_0 = \{ \rho_\mu: \rho_\mu(x) = \mu x, \, \rho_\mu(y) = \mu y, \, \mu \in \Bbbk^\ast \} 
\]
and
\[
\Ct_1 = \{\eta_\mu: \eta_\mu(x) = \mu \alpha y, \, \eta_\mu(y) = \mu \beta x, \, \mu \in \Bbbk^\ast \}.
\]
\end{enumerate}

In particular, $\Cts$ is abelian when $\sigma$ is a flip
automorphism or a toric automorphism with $\alpha\neq\beta$.
\end{lemma}

\begin{proof}
The result follows from a direct comparison of $\varphi \sigma$ and $\sigma \varphi$, depending on whether $\varphi \in \aut(\mathcal A_{-1})$ is a toric or a flip automorphism.
\end{proof}

\subsection{The isotropy group $\aut_{\delta}(\mathcal A_{-1})$ when $\sigma$ is toric}

Suppose that $\sigma \in \aut(\mathcal A_{-1})$ is toric. If we restrict to the toric part of the centralizer, then the isotropy is
described exactly as in the case $q\neq\pm1$, with $t=2$. More precisely, the families appearing in the classification contribute with character equations in $(\mu_1,\mu_2)\in{\Bbbk^\ast}^2$. In particular, the same arithmetic criterion, applied to the exponent vectors in $\Gamma(\delta)$, determines the finiteness of the toric part of the isotropy group.

However, when $\alpha=\beta$, the centralizer also contains flip automorphisms. These automorphisms may interchange the subspaces $\Bbbk \, I_{y^j x^i}$ and $\Bbbk \, I_{y^i x^j}$, the families $\mathcal P_x$ and $\mathcal P_y$, and also the roles of $D_x$ and $D_y$. Therefore, the full isotropy group is not determined only by $\Gamma(\delta)$; one must also impose compatibility conditions on the coefficients of the corresponding components.

\begin{proposition} \label{independent2}
	Let $\sigma$ be the toric automorphism given by $\sigma(x)=\gamma x$, $\sigma(y)=\gamma y$. Then we can decompose $\der_\sigma(\mathcal{A}_{-1})$ as a direct sum of $\aut(\QPs)$-submodules as follows:
	\begin{enumerate}
			
		\item[(1)] the \textbf{inner family} is the direct sum $\bigoplus\limits_{i > j \geq 0} \left( \Bbbk \, I_{y^j x^i} + \Bbbk \, I_{y^i x^j} \right) \oplus \bigoplus\limits_{j \geq 0} \Bbbk \, I_{y^j x^j}$, where
        \begin{eqnarray} \label{eqinnerneqzero}
\left\{ \begin{array}{lll}
	I_{y^j x^i} = 0 \text{ if and only if } i, j \text{ are even } & , & \text{ if } \gamma = 1 \\ [0.2cm]
	I_{y^j x^i} = 0 \text{ if and only if } i, j \text{ are odd } & , & \text{ if } \gamma = -1 \\ [0.2cm]
	I_{y^j x^i} \neq 0 \text{ for all } i, j & , & \text{ if } \gamma \neq \pm 1
\end{array} \right. \; ;
\end{eqnarray}

        \item[(2)] (if $\gamma = -1$) the \textbf{$\mathcal P_x + \mathcal P_y$-family} is the direct sum $\bigoplus\limits_{k \geq 0} \left( \Bbbk \, \mathcal{X}_k + \Bbbk \, \mathcal{Y}_k \right)$, where
		\[
		\left\{ \begin{array}{l}
			\mathcal{X}_k(x) = y^k \\
			\mathcal{X}_k(y) = 0
		\end{array} \right.
		\qquad \text{and} \qquad
		\left\{ \begin{array}{l}
			\mathcal{Y}_k(x) = 0 \\
			\mathcal{Y}_k(y) = x^k
		\end{array} \right. \, ;
		\]

        \item[(3)] (if $\gamma = \pm 1$) the \textbf{$\lambda_1 D_x + \lambda_2 D_y$-family} is the direct sum
        \[
        \bigoplus\limits_{r,s \geq 0} \left( \Bbbk \, y^{2s} x^{2r} D_x  + \Bbbk y^{2r} x^{2s} D_y \right).
        \]

    \end{enumerate}
\end{proposition}

\begin{proof}
    For toric automorphisms of $\mathcal{A}_{-1}$, the proof that each summand is stable by the action is the same as that of Proposition \ref{independent}.
    
    For a flip automorphism $\eta(x) = \nu_1 y$, $\eta(y) = \nu_2 x$ of $\QPs$, we have, for every $i,j$,
    \begin{eqnarray} \label{eqactflipinn}
        \eta I_{y^j x^i} \eta^{-1} = I_{\eta(y^j x^i)} = I_{\nu_1^i \nu_2^j \, x^j y^i} = (-1)^{ij} \nu_1^i \nu_2^j \, I_{y^i x^j}
    \end{eqnarray}
    with
    \[
		\left\{ \begin{array}{l}
			I_{y^j x^i}(x) = y^j x^i \sigma(x) - x y^j x^i = \left( \gamma - (-1)^{j} \right) y^j x^{i+1} \\[0.2cm]
			I_{y^j x^i}(y) = y^j x^i \sigma(y) - y y^j x^i = \left( \gamma - (-1)^{i} \right) y^j x^{i}y
		\end{array} \right. \; .
		\]
    So item (1) follows.

    Regarding (2) and (3), note that $\eta^{-1}(x) = \nu_2^{-1} y$ and $\eta^{-1}(y) = \nu_1^{-1} x$. So, it is easy to see that, for each $k \geq 0$, we have
    \begin{eqnarray} \label{eqXkYkbyflip}
        \eta \mathcal{X}_k \eta^{-1} = \nu_1^{-1} \nu_2^k \mathcal{Y}_k \qquad \text{ and } \qquad \eta \mathcal{Y}_k \eta^{-1} = \nu_1^k \nu_2^{-1} \mathcal{X}_k \,
    \end{eqnarray}
    and, for every $r,s \geq 0$, we have
    \begin{eqnarray} \label{eqactionDflip1}
        \left\{ \begin{array}{l}
		\eta \left( y^{2s} x^{2r} D_x \right) \eta^{-1} = \nu_1^{2r} \nu_2^{2s} \, y^{2r} x^{2s} D_y \\[0.3cm]
        \eta \left( y^{2r} x^{2s} D_y \right) \eta^{-1} = \nu_1^{2s} \nu_2^{2r} \, y^{2s} x^{2r} D_x
	\end{array} \right. \; , \text{ if } \gamma = 1 ,
    \end{eqnarray}
    and
    \begin{eqnarray} \label{eqactionDflip2}
        \left\{ \begin{array}{l}
		\eta \left( y^{2s} x^{2r} D_x \right) \eta^{-1} = - \nu_1^{2r+1} \nu_2^{2s+1} \, y^{2r} x^{2s} D_y \\[0.3cm]
        \eta \left( y^{2r} x^{2s} D_y \right) \eta^{-1} = - \nu_1^{2s+1} \nu_2^{2r+1} \, y^{2s} x^{2r} D_x
	\end{array} \right. \; , \text{ if } \gamma = -1 ,
    \end{eqnarray}
    This completes the proof.
\end{proof}

We now determine when a flip automorphism belongs to the isotropy group in the case that $\sigma$ is toric with $\alpha = \beta$.

\begin{proposition} \label{isotropy1}
Let $\sigma$ be the automorphism given by $\sigma(x) = \gamma x$, $\sigma(y) = \gamma y$ and let $\eta(x) = \nu_1 y$, $\eta(y) = \nu_2 x$ be a flip automorphism.

\begin{enumerate}

\item[(1)] The \textbf{inner family}.

If $I_{w} = [w, \_]_{\sigma}$ is the inner $\sigma$-derivation induced by $w = \sum c_{ij} y^j x^i$, then $\eta \in \aut_{I_w}(\mathcal A_{-1})$ if and only if
\[
c_{ji} = (-1)^{ij} \nu_1^i \nu_2^j c_{ij}, \text{ for all } i,j \text{ such that } c_{ij} I_{y^j x^i} \neq 0.
\]
In particular, whenever $c_{ij} I_{y^j x^i} \neq 0$ we must have
\[
\left\{ \begin{array}{ll}
     (\nu_1 \nu_2)^{i+j} = 1, & \text{if } i \neq j \\[0.1cm]
     (\nu_1 \nu_2)^j = (-1)^{j}, & \text{if } i = j
\end{array} \right. \; .
\]

\item[(2)] The \textbf{$\mathcal P_x + \mathcal P_y$-family.}

If $\gamma=-1$ and 
\[
\left\{ \begin{array}{l}
	\mathcal{P}_x(x) = a(y) \; \in \, \Bbbk[y] \\ [0.1cm]
	\mathcal{P}_x(y) = 0
	\end{array} \right. \qquad , \qquad
    \left\{ \begin{array}{l}
	\mathcal{P}_y(x) = 0 \\ [0.1cm]
	\mathcal{P}_y(y) = b(x) \; \in \, \Bbbk[x]
	\end{array} \right.
\]
then $\eta \in \aut_{\mathcal P_x + \mathcal P_y}(\mathcal A_{-1})$ if and only if
\[
a(y) = \nu_2^{-1} b(\nu_1 y)
\quad \text{and} \quad
b(x) = \nu_1^{-1} a(\nu_2 x).
\]
That is, if $a(y) = \sum a_k y^k$ and $b(x) = \sum b_k x^k$, then, for all $k$,
\[
a_k = \nu_2^{-1} \nu_1^k b_k
\quad \text{and} \quad
b_k = \nu_1^{-1} \nu_2^k a_k \;.
\]
In particular, we must have $(\nu_1 \nu_2)^{k-1} = 1$ whenever $a_k \neq 0$ (and $b_k \neq 0$).

\item[(3)] The \textbf{$\lambda_1 D_x + \lambda_2 D_y$-family.}

If $\gamma = \pm 1$ then $\eta \in \aut_{\lambda_1 D_x + \lambda_2 D_y}(\mathcal A_{-1})$ if and only if,
\[
\left\{ \begin{array}{clcll}
	\lambda_1 = \eta(\lambda_2) & \text{and} & \lambda_2 = \eta(\lambda_1) & , & \text{ if } \gamma = 1 \\ [0.3cm]
	\lambda_1 = - \nu_1 \nu_2 \, \eta(\lambda_2) & \text{and} & \lambda_2 = - \nu_1 \nu_2 \, \eta(\lambda_1) & , & \text{ if } \gamma = -1
\end{array} \right. \; .
\]

That is, if $\lambda_1 = \sum c_{rs} y^{2s} x^{2r}$ and $\lambda_2 = \sum d_{rs} y^{2s} x^{2r}$, then, for all $r,s$,
\[
\left\{ \begin{array}{clcll}
	c_{rs} = \nu_1^{2s} \nu_2^{2r} d_{sr} & \text{and} & d_{sr} = \nu_1^{2r} \nu_2^{2s} c_{rs} & , & \text{ if } \gamma = 1 \\ [0.3cm]
	c_{rs} = - \nu_1^{2s+1} \nu_2^{2r+1} d_{sr} & \text{and} & d_{sr} = - \nu_1^{2r+1} \nu_2^{2s+1} c_{rs} & , & \text{ if } \gamma = -1
\end{array} \right. \; .
\]
In particular, we must have
\[
\left\{ \begin{array}{clcll}
	(\nu_1 \nu_2)^{2(r+s)} = 1 & \text{whenever} & c_{rs} \neq 0 & , & \text{ if } \gamma = 1 \\ [0.3cm]
	(\nu_1 \nu_2)^{2(r+s+1)} = 1 & \text{whenever} & c_{rs} \neq 0 & , & \text{ if } \gamma = -1
\end{array} \right. \; .
\]

\end{enumerate}
\end{proposition}

\begin{proof}
In (1) we have $I_{w} = \sum c_{ij} I_{y^j x^i}$. By Proposition \ref{independent2} and Lemma \ref{lemaindep},
\[
\eta I_w \eta^{-1} = I_w \quad \Longleftrightarrow \quad \eta \left( c_{ij} I_{y^j x^i} + c_{ji} I_{y^i x^j} \right) \eta^{-1} = c_{ij} I_{y^j x^i} + c_{ji} I_{y^i x^j}, \; \forall \, i,j
\]
(regardless of whether $i \neq j$ or $i=j$). But, by (\ref{eqactflipinn}), this occurs if and only if
\[
c_{ji} = (-1)^{ij} \nu_1^i \nu_2^j c_{ij} \text{ for all } i,j \text{ such that } c_{ij} I_{y^j x^i} \neq 0
\]
because, by (\ref{eqinnerneqzero}), $I_{y^j x^i} \neq 0$ if and only if $I_{y^i x^j} \neq 0$. This proves (1).

Analogously (2) and (3) follow from equations (\ref{eqXkYkbyflip}), (\ref{eqactionDflip1}) and (\ref{eqactionDflip2}) since in (2) we have
\[
\mathcal{P}_x + \mathcal{P}_y = \sum \left( a_k \mathcal{X}_k + b_k \mathcal{Y}_k \right)
\]
and in (3)
\[
\lambda_1 D_x + \lambda_2 D_y = \sum \left( c_{rs} y^{2s} x^{2r} D_x + d_{sr} y^{2r} x^{2s} D_y \right).
\]
\end{proof}

For an arbitrary $\sigma$-derivation
\[
\delta
=
[w,\_]_\sigma+\mathcal P_x+\mathcal P_y+\lambda_1D_x+\lambda_2D_y \, ,
\]
a flip automorphism $\eta \in \aut_\delta(\mathcal{A}_{-1})$ if and only if it satisfies all the conditions above. As in Proposition \ref{innercentersima}, we can reformulate
Proposition \ref{isotropy1}(1) in terms of the twisted
center as follows.

\begin{proposition} \label{innercentersima1}
    If $\sigma(x) = \gamma x$, $\sigma(y) = \gamma y$, then a flip automorphism $\eta \in \aut_{I_w}(\QPs)$ if and only if $\eta(w) - w \in \mathcal Z_\sigma(\QPs)$, where, by (\ref{eqinnerneqzero}),
    \[
    \mathcal Z_\sigma(\QPs) = 
    \begin{cases}
    \Bbbk[x^2,y^2], & \gamma = 1, \\[0.1cm]
    \Bbbk[x^2,y^2]\,yx, & \gamma = -1, \\[0.1cm]
    0, & \gamma \neq \pm 1.
    \end{cases}
    \]
\end{proposition}

The following corollary gives an explicit description
of the isotropy groups of ordinary derivations of
$\mathcal A_{-1}$, a case not treated in \cite{SBVA}.

\begin{corollary}\label{ordinary-1}
Let $\delta\in\der(\mathcal A_{-1})$ be an ordinary derivation.
Then
\[
\delta=[w,\_]+\lambda_1D_x+\lambda_2D_y,
\]
where $\lambda_1, \lambda_2 \in \centro(\QPs) = \Bbbk[x^2,y^2]$ and $\left\{ \begin{array}{l}
D_x(x) = x \\
D_x(y) = 0
\end{array} \right. , \; \left\{ \begin{array}{l}
D_y(x) = 0 \\
D_y(y) = y
\end{array} \right.$.

Furthermore, 

\begin{enumerate}
\item[(i)] a toric automorphism $\rho(x) = \mu_1 x$, $\rho(y) = \mu_2 y$ belongs to isotropy group of $\delta$ if and only if $\rho(w) - w \in \centro(\QPs)$ and $\rho(\lambda_1) = \lambda_1$, $\rho(\lambda_2) = \lambda_2$;

\item[(ii)] a flip automorphism $\eta(x) = \nu_1 y$, $\eta(y) = \nu_2 x$ belongs to isotropy group of $\delta$ if and only if $\eta(w) - w \in \centro(\QPs)$ and $\eta(\lambda_2) = \lambda_1$, $\eta(\lambda_1) = \lambda_2$.
\end{enumerate}
\end{corollary}

\begin{proof}
This follows from Theorem \ref{classification-1}, Proposition \ref{isotropy} and Proposition \ref{isotropy1} by taking $\sigma=\mathrm{id}$.
\end{proof}

\subsection{The isotropy group $\aut_{\delta}(\mathcal A_{-1})$ when $\sigma$ is flip}

We now study the cases where $\sigma(x) = \alpha y$, $\sigma(y) = \beta x$ is a flip automorphism. We shall see below that in this case we can also decompose $\der_\sigma(\QPs)$ into a direct sum of $\Cts$-submodules, where the centralizer $\Cts = \Ct_0 \cup \Ct_1$ is given by Lemma \ref{centralizador}.

According to Theorem \ref{classification-1}, we can decompose each $\sigma$-derivation as
\[
\delta = [w,\_]_{\sigma} + \sum\limits_{\alpha \beta = (-1)^k} \varepsilon_k \dk
\]
for some $w \in \QPs$, and where each $\varepsilon_k \in \Bbbk$ and each $\dk$ is the (non inner) $\sigma$-derivation given by
\begin{equation} \tag{\ref{derdk}}
	\left\{ \begin{array}{l}
		\dk(x) = - \beta^{-1} \sum\limits_{l=0}^{k} (-1)^l \beta^{k-l} y^l x^{k-l} \\ [0.3cm]
		\dk(y) = \sum\limits_{l=0}^{k} \beta^{k-l} y^l x^{k-l}
	\end{array} \right. .
\end{equation}
In particular, if $\alpha \beta \neq \pm 1$ then every $\sigma$-derivation is inner.

As mentioned above (and proved below), each subspace $\Bbbk \dk$ is stable over the action of $\Cts$. However, for the inner $\sigma$-derivations, the ``canonical'' subspaces $\Bbbk I_{y^j x^i}$ are not stable over actions of flip automorphism. So, it is convenient to change the basis $\{ y^j x^i \}$ of $\QPs$ according to the parameters $\alpha$ and $\beta$.

For every $i, j \geq 0$, fix a square root of $\alpha^{i} \beta^{j}$ and denote it by $\xi_{ij}$. Then, for each $i > j \geq 0$, define
\begin{eqnarray} \label{uji}
    \left\{ \begin{array}{l}
		u_{ji} = \xi_{ji} y^j x^i + \xi_{ij} y^i x^j \\ [0.1cm]
		u_{ij} = \xi_{ji} y^j x^i - \xi_{ij} y^i x^j \\ [0.1cm]
        u_{jj} = 2 \xi_{jj} y^j x^j
	\end{array} \right. .
\end{eqnarray}
For each $i > j$,
\[
y^j x^i = (2 \xi_{ji})^{-1}(u_{ji} + u_{ij}) \qquad \text{and} \qquad y^i x^j = (2 \xi_{ij})^{-1}(u_{ji} - u_{ij}) ,
\]
so every $w \in \QPs$ can be written uniquely of the form $w = \sum C_{ji} u_{ji}$. That is, the set $\{ u_{ij} \}$ is a $\Bbbk$-basis of $\QPs$ (which depends on the choice of the square roots $\xi_{ij}$ of $\alpha^{i} \beta^{j}$).

Moreover, if $f\in\centroT(\QPs)$, then, since $x^2$ is central, $f x^2 = x^2 f = f \sigma(x^2) = f \alpha^2 y^2$. Since $\QPs$ is a domain and $x^2-\alpha^2y^2\neq0$, it follows that $f=0$. Thus $\centroT(\QPs) = \{0\}$. Consequently, the linear map $w\mapsto I_w$ is injective, and the set $\{ I_{u_{ij}} \}$, associated with the basis $\{u_{ij}\}$, is a basis of the space of inner $\sigma$-derivations.

\begin{proposition} \label{independent3}
	Let $\sigma(x) = \alpha y$, $\sigma(y) = \beta x$ be a flip automorphism of $\QPs$. Then, we can decompose $\der_\sigma(\QPs)$ as a direct sum of $\Cts$-submodules of dimension $1$ as follows:
	\begin{enumerate}
			
		\item[(1)] the \textbf{inner family} is the direct sum $\bigoplus\limits_{{i,j \geq 0}} \, \Bbbk \, I_{u_{ji}}$, where each $u_{ji}$ is given by (\ref{uji});

        \item[(2)] (if $\alpha \beta = \pm 1$) the \textbf{$\varepsilon_k \dk$-family} is the direct sum $\bigoplus\limits_{\alpha \beta = (-1)^k} \!\! \Bbbk \dk$.
        
    \end{enumerate}
\end{proposition}

\begin{proof}
    (1) For a toric automorphism $\rho_{\mu}(x) = \mu x$, $\rho_{\mu}(y) = \mu y$ in $\Ct_0$, it is easy to see that $\rho(u_{ji}) = \mu^{i+j} u_{ji}$, so
    \begin{eqnarray} \label{actiontoricuji}
        \rho_{\mu} I_{u_{ji}} \rho_{\mu}^{-1} = I_{\rho_{\mu}(u_{ji})} = I_{\mu^{i+j} u_{ji}} = \mu^{i+j} \, I_{u_{ji}} \; \in \; \Bbbk \, I_{u_{ji}} \; , \forall \, i,j.
    \end{eqnarray}
    For a flip automorphism $\eta_{\mu}(x) = \mu \alpha y$, $\eta_{\mu}(y) = \mu \beta x$ in $\Ct_1$, denote by $\zeta_{ij} = \xi_{ij} \xi_{ji}$. Note that each $\zeta_{ij} = \zeta_{ji}$ is a square root of $(\alpha \beta)^{i+j}$ and, for every $i, j$,
    \begin{eqnarray*} \label{sqralphabeta}
    	\xi_{ij} \alpha^j \beta^i \; = \; \xi_{ij} \xi_{ji}^2 \; = \; \zeta_{ij} \xi_{ji}.
    \end{eqnarray*}
    With this, it is easy to see that for every $i > j \geq 0$, we have
    \begin{eqnarray*}
        \left\{ \begin{array}{l}
             \eta_{\mu} (u_{ji}) = (-1)^{ij} \mu^{i+j} \zeta_{ij} u_{ji} \\[0.1cm]
             \eta_{\mu} (u_{ij}) = (-1)^{ij+1} \mu^{i+j} \zeta_{ij} u_{ij} \\[0.1cm]
             \eta_{\mu}(u_{jj}) = (-1)^j \mu^{2j} \zeta_{jj} u_{jj}
        \end{array}
        \right.
    \end{eqnarray*}
    which implies
    \begin{eqnarray} \label{actionuijbyflip}
        \left\{
        \begin{array}{l}
             \eta_{\mu} I_{u_{ji}} \eta_{\mu}^{-1} = (-1)^{ij} \mu^{i+j} \zeta_{ij} \, I_{u_{ji}} \\[0.2cm]
             \eta_{\mu} I_{u_{ij}} \eta_{\mu}^{-1} = - (-1)^{ij} \mu^{i+j} \zeta_{ij} \, I_{u_{ij}} \\[0.2cm]
             \eta_{\mu} I_{u_{jj}} \eta_{\mu}^{-1} = (-1)^{j} \mu^{2j} \zeta_{jj} \, I_{u_{jj}}
        \end{array}
        \right. \; .
    \end{eqnarray}
    So (1) follows.
    
    Regarding (2), suppose that $\alpha \beta = (-1)^k$ and let $\dk$ be the $\sigma$-derivation given by (\ref{derdk}). It is easy to see that, for a toric automorphism $\rho_{\mu} \in \Ct_0$ and a flip automorphism $\eta_{\mu} \in \Ct_1$ we have
    \begin{eqnarray} \label{actintoricflipDk}
        \left\{
        \begin{array}{l}
             \rho_{\mu} \dk \rho_{\mu}^{-1} = \mu^{k-1} \dk \\
             \eta_{\mu} \dk \eta_{\mu}^{-1} = - \mu^{k-1} \dk
        \end{array}
        \right.
    \end{eqnarray}
    This completes the proof.
\end{proof}

We now determine when an automorphism belongs to the isotropy group of a $\sigma$-derivation. We divide it into two parts because when $\rho_{\mu} \in \Ct_0$ is a toric automorphism, for the analysis of the action over the inner part $[w,\_]_{\sigma}$ it is sufficient to decompose $w$ into homogeneous components. However, when $\eta_\mu \in \Ct_1$ is a flip automorphism, it is convenient to decompose $w$ as a linear combination of the $u_{ji}$'s.

\begin{proposition} 
Let $\sigma(x)=\alpha y$, $\sigma(y)=\beta x$ be a flip automorphism of $\QPs$ and let
\[
\delta = [w,\_]_{\sigma} + \sum\limits_{\alpha \beta = (-1)^k} \varepsilon_k \dk
\]
be a $\sigma$-derivation. Write $w = \sum w_k$ where each $w_k = \epsilon_1 x^{k-1} + \epsilon_2 y x^{k-2} + \epsilon_3 y^2 x^{k-3} + \ldots + \epsilon_k y^{k-1}$ is the homogeneous component of $w$ of degree $k-1$.

If $\rho_\mu \in \Ct_0$ is a toric automorphism of the centralizer $\Cts$, then $\rho_\mu \in \aut_\delta(\mathcal A_{-1})$ if and only if $\mu^{k-1} = 1$, for every $k$ such that $w_k \neq 0$ or $\varepsilon_k \neq 0$.
\end{proposition}

\begin{proof}
    Since each $w_k$ belongs to $\sum_{i+j=k-1} \Bbbk u_{ji}$, the result follows from Lemma \ref{lemaindep} and Proposition \ref{independent3}, equations (\ref{actiontoricuji}) and (\ref{actintoricflipDk}).
\end{proof}

Analogously, by equations (\ref{actionuijbyflip}) and (\ref{actintoricflipDk}) we have the following.

\begin{proposition} 
Let $\sigma(x)=\alpha y$, $\sigma(y)=\beta x$ be a flip automorphism of $\QPs$ and let
\[
\delta = [w,\_]_{\sigma} + \sum\limits_{\alpha \beta = (-1)^k} \varepsilon_k \dk
\]
be a $\sigma$-derivation. Fix a basis $\{ u_{ji} \}$ of $\QPs$ as in (\ref{uji}), write $w = \sum C_{ji} u_{ji}$ as a $\Bbbk$-linear combination of this basis and, for each $i,j$, let $\zeta_{ij} = \zeta_{ji} = \xi_{ij} \xi_{ji}$ be the square root of $(\alpha \beta)^{i+j}$ associated with this choice.
   
If $\eta_\mu \in \Ct_1$ is a flip automorphism of the centralizer $\Cts$, then $\eta_\mu \in \aut_\delta(\QPs)$ if and only if it satisfies the following two conditions.

\begin{enumerate}

\item[(1)] The \textbf{inner family}: for every $C_{ji} \neq 0$, we must have
\begin{eqnarray*}
    \left\{ \begin{array}{lll}
		\mu^{i+j} = (-1)^{ij} \zeta_{ij}^{-1} & , & \text{if } i \geq j \\ [0.1cm]
        \mu^{i+j} = -(-1)^{ij} \zeta_{ij}^{-1} & , & \text{if } j > i
	\end{array} \right. \; .
\end{eqnarray*}
In particular, if $i \neq j$ we must have $C_{ij} = 0$ or $C_{ji} = 0$.

\item[(2)] The \textbf{$\varepsilon_k \dk$-family}: for every $\varepsilon_k \neq 0$ we must have $- \mu^{k-1} = 1$.

\end{enumerate}

\end{proposition}

\begin{remark}
The twisted center determines when two elements induce
the same inner $\sigma$-derivation: $I_w=I_v$ if and
only if $w-v\in\centroT(\QPq)$. Moreover, for every
$\rho\in C_{\aut(\QPq)}(\sigma)$, the identity
$\rho I_w\rho^{-1}=I_{\rho(w)}$ gives
\[
\rho\in\aut_{I_w}(\QPq)
\iff
\rho(w)-w\in\centroT(\QPq).
\]

This provides a common interpretation of Propositions
\ref{innercentersima} and \ref{innercentersima1}.
In Proposition \ref{pho(x)}, removing the terms in the
twisted center produces an element $w'$ such that
$I_w=I_{w'}$ and the isotropy condition becomes
$\rho(w')=w'$.
When $q=-1$ and $\sigma$ is a flip automorphism,
$\centroT(\QPs)=\{0\}$. Consequently, $I_w=I_v$ if
and only if $w=v$, so each inner $\sigma$-derivation
is induced by a unique element of $\QPs$.
For every $\rho\in\Cts$, we therefore have
\[
\rho I_w\rho^{-1}=I_w
\iff
I_{\rho(w)}=I_w
\iff
\rho(w)=w.
\]
This holds even though the ordinary center
$\centro(\QPs)=\Bbbk[x^2,y^2]$ is nontrivial.

The same principle applies to ordinary inner derivations
on the differential Ore extension
$A_h=\Bbbk[x][t;d]$, where $d=h(x)\partial_x$ and
$\deg h\geq1$. Writing $I_w(a)=wa-aw$ for $a\in A_h$,
the equality $\centro(A_h)=\Bbbk$ gives
\[
\rho\in\aut_{I_w}(A_h) \iff \rho(w)-w\in\Bbbk, \quad \rho\in\aut(A_h).
\]

In this case, $w$ is determined by $I_w$ only up to an additive
constant. For example, the automorphism $\rho(x)=x$,
$\rho(t)=t+1$ fixes $I_t$, since $I_{t+1}=I_t$,
although $\rho(t)\neq t$.
The results in \cite[Propositions 25, 27 and 28]{BaltazarSilvaMartini2026} give conditions ensuring that an automorphism $\rho\in\aut_{I_w}(A_h)$ also satisfies $\rho(w)=w$.
\end{remark}

\section*{Acknowledgements}

The authors are grateful to Prof. Tomasz Brzeziński for suggesting the initial references that inspired this work and for his valuable encouragement. The authors also thank Prof. David A. Jordan for sharing his recent work with us
and for the valuable discussion on the subject after the first version of this paper.

\bibliographystyle{abbrv}

\bibliography{ref}

@article{AlevChamarie,
  author  = {Alev, J. and Chamarie, M.},
  title   = {D{\'e}rivations et automorphismes de quelques alg{\`e}bres quantiques},
  journal = {Communications in Algebra},
  volume  = {20},
  number  = {6},
  pages   = {1787--1802},
  year    = {1992}
}

@book{Borel,
  author    = {Borel, A.},
  title     = {Linear Algebraic Groups},
  publisher = {Springer},
  year      = {1991}
}

@article{BaltazarSilvaMartini2026,
  author  = {Rene Baltazar and Leonardo Duarte Silva and Grasiela Martini},
  title   = {On the isotropy of differential {Ore} extensions},
  journal = {preprint: arXiv:2604.17161},
  year    = {2026},
  eprint  = {2604.17161},
  archivePrefix = {arXiv},
  primaryClass  = {math.RA}
}

@article{Jordan2025,
  author  = {Jordan, David A.},
  title   = {Skew derivations of quantum tori and quantum affine spaces},
  journal = {Journal of Algebra},
  volume  = {672},
  pages   = {261--302},
  year    = {2025},
  doi     = {10.1016/j.jalgebra.2025.02.032},
  url     = {https://doi.org/10.1016/j.jalgebra.2025.02.032}
}

@article{Brz2018,
  author  = {Almulhem, M. and Brzezi{\'n}ski, T.},
  title   = {Skew derivations on generalized {Weyl} algebras},
  journal = {Journal of Algebra},
  volume  = {493},
  pages   = {194-235},
  year    = {2018},
  doi     = {10.1016/j.jalgebra.2017.09.018}
}

@article{SBVA,
  author  = {Baltazar, R. and Santana, A.  and Vinciguerra, R. and Araujo, W.},
  title   = {On isotropy groups of quantum plane},
  journal = {Journal of Pure and Applied Algebra},
  volume  = {229},
  number  = {11},
  pages   = {},
  year    = {2025}
}

@article{PanBalta,
  author  = {Baltazar, R. and Pan, I.},
  title   = {On the automorphism group of a polynomial differential ring in two variables},
  journal = {Journal of Algebra},
  volume  = {576},
  pages   = {197--227},
  year    = {2021},
  doi     = {10.1016/j.jalgebra.2021.02.009}
}

@article{YanShamsuddin,
  author  = {Yan, D.},
  title   = {On {Shamsuddin} derivations and the isotropy groups},
  journal = {Journal of Algebra},
  volume  = {637},
  pages   = {243--252},
  year    = {2024},
  doi     = {10.1016/j.jalgebra.2023.09.023}
}

@article{PanMendes,
  author  = {Mendes, L. G. and Pan, I.},
  title   = {On plane polynomial automorphisms commuting with simple derivations},
  journal = {Journal of Pure and Applied Algebra},
  volume  = {221},
  number  = {4},
  pages   = {875--882},
  year    = {2017},
  doi     = {10.1016/j.jpaa.2016.08.009}
}

\bigskip

\noindent
\textsc{Rene Baltazar}\\
Universidade Federal do Rio Grande - FURG\\
Santo Antônio da Patrulha/RS, Brasil\\
\texttt{renebaltazar.furg@gmail.com}

\bigskip

\noindent
\textsc{Rafael Cavalheiro}\\
Universidade Federal do Rio Grande - FURG\\
Santo Antônio da Patrulha/RS, Brasil\\
\texttt{rcavalheiro@furg.br}

\end{document}